\documentclass[11pt]{amsart}
\usepackage[margin=1.35in]{geometry}
\usepackage{amscd,amsmath,amsxtra,amsthm,amssymb,stmaryrd,xr,mathrsfs,mathtools,enumerate,commath, comment, enumitem}
\usepackage{stmaryrd}
\usepackage{xcolor}
\usepackage{commath}
\usepackage{comment}
\usepackage{tikz-cd}
\usepackage{longtable} 
\usepackage{pdflscape} 
\usepackage{booktabs}
\usepackage{hyperref}
\definecolor{vegasgold}{rgb}{0.77, 0.7, 0.35}
\definecolor{darkgoldenrod}{rgb}{0.72, 0.53, 0.04}
\definecolor{gold(metallic)}{rgb}{0.83, 0.69, 0.22}
\hypersetup{
 colorlinks=true,
 linkcolor=darkgoldenrod,
 filecolor=brown,      
 urlcolor=gold(metallic),
 citecolor=darkgoldenrod,
 }

\usepackage[all,cmtip]{xy}
\usepackage{orcidlink}

\DeclareFontFamily{U}{wncy}{}
\DeclareFontShape{U}{wncy}{m}{n}{<->wncyr10}{}
\DeclareSymbolFont{mcy}{U}{wncy}{m}{n}
\DeclareMathSymbol{\Sh}{\mathord}{mcy}{"58}
\usepackage[T2A,T1]{fontenc}
\usepackage[OT2,T1]{fontenc}

\newtheorem{theorem}{Theorem}[section]
\newtheorem{lemma}[theorem]{Lemma}
\newtheorem{ass}[theorem]{Assumption}
\newtheorem*{theorem*}{Theorem}
\newtheorem*{ass*}{Assumption}
\newtheorem{definition}[theorem]{Definition}
\newtheorem{corollary}[theorem]{Corollary}
\newtheorem{remark}[theorem]{Remark}

\newtheorem{conjecture}[theorem]{Conjecture}
\newtheorem{proposition}[theorem]{Proposition}
\newtheorem{question}[theorem]{Question}

\newcommand{\cF}{\mathcal{F}}

\newcommand{\cK}{\mathcal{K}}

\newcommand{\cD}{\mathcal{D}}

\newcommand{\cU}{\mathcal{U}}

\newcommand{\LcO}{\Lambda_{\cO}}

\newcommand{\Z}{\mathbb{Z}}
\newcommand{\p}{\mathfrak{p}}
\newcommand{\Q}{\mathbb{Q}}

\newcommand{\cO}{\mathcal{O}}

\newcommand{\Hom}{\mathrm{Hom}}

\newcommand{\op}[1]{\operatorname{#1}}

\newcommand\mtx[4] { \left( {\begin{array}{cc}
 #1 & #2 \\
 #3 & #4 \\
 \end{array} } \right)}

\numberwithin{equation}{section}

\begin{document}

\title[Iwasawa invariants of Artin representations]{On the Iwasawa invariants of Artin representations}
\author[A.~Karnataki]{Aditya Karnataki\, \orcidlink{0000-0002-0849-5672}}
\address[Karnataki]{Chennai Mathematical Institute, H1, SIPCOT IT Park, Kelambakkam, Siruseri, Tamil Nadu 603103, India}
\email{adityak@cmi.ac.in}

\author[A.~Ray]{Anwesh Ray\, \orcidlink{0000-0001-6946-1559}}
\address[Ray]{Chennai Mathematical Institute, H1, SIPCOT IT Park, Kelambakkam, Siruseri, Tamil Nadu 603103, India}
\email{anwesh@cmi.ac.in}


\maketitle

\begin{abstract}
 We study Iwasawa invariants associated to Selmer groups of Artin representations, and criteria for the vanishing of the associated algebraic Iwasawa invariants. The conditions obtained can be used to study natural distribution questions in this context.
\end{abstract}

\section{Introduction}
\par Let $p$ be an odd prime number and $\Z_p$ denote the ring of $p$-adic integers. Let $K$ be a number field and fix an algebraic closure $\bar{K}$ of $K$. A $\Z_p$-extension of $K$ is an infinite Galois extension $K_\infty/K$ such that the Galois group $\op{Gal}(K_\infty/K)$ is isomorphic to $\Z_p$ as a topological group. Let $K_n\subset K_\infty$ be the extension of $K$ for which $\op{Gal}(K_n/K)\simeq \Z/p^n\Z$, and $h_k(K_n)$ denote the cardinality of the $p$-primary part of the class number of $K_n$. Writing $h_p(K_n)=p^{e_n}$, Iwasawa \cite{iwasawa1959gamma} showed that for all large enough values of $n$,
\[e_n=p^n \mu+n \lambda+\nu,\] where $\mu, \lambda\in \Z_{\geq 0}$ and $\nu \in \Z$ are invariants that depend on the extension $K_\infty/K$ and not on $n$. The results of Iwasawa motivated Mazur \cite{mazur1972rational} to study the growth properties of the $p$-primary Selmer groups of $p$-ordinary abelian varieties in $\Z_p$-extensions. Mazur's results were later extended to very general classes of ordinary Galois representations by Greenberg, cf. \cite{greenberg1989iwasawa}. \par Another class of representations that are natural to consider are Artin representations. Let $K/\Q$ be a finite and totally imaginary Galois extension with Galois group $\Delta:=\op{Gal}(K/\Q)$, and let $\rho:\Delta \rightarrow \op{GL}_d(\bar{\Q})$ be an irreducible Artin representation. Let $p$ be an odd prime number and let $\sigma_p: \bar{\Q}\hookrightarrow \bar{\Q}_p$ be an embedding and via $\sigma_p$ we view $\rho$ as a representation $\rho: \Delta\rightarrow \op{GL}_d(\bar{\Q}_p)$. 
Let $v$ denote an archimedian prime of $K$, and set $K_v$ to denote the $v$-adic completion of $K$. We shall identify $\op{Gal}(K_v/\mathbb{R})$ with a subgroup $\Delta_v$ of $\Delta$. Set $d^+$ to denote the multiplicity of the trivial character in $\rho_{|\Delta_v}$ and observe that this number is well defined and independent of the choice of the archimedian prime $v$. When $K$ is totally real, we find that $d^+=d$. Let $\p$ be a prime of $K$ that lies above $p$, and let $\Delta_{\p}\subset \Delta$ be the decomposition group at $\p$. Following Greenberg and Vatsal \cite{greenbergvatsalArtinrepns}, we make the following assumptions.
\begin{ass}\label{main ass}
    With respect to the notation above, assume that 
    \begin{enumerate}
        \item $p$ does not divide $[K:\Q]$, 
        \item $d^+=1$, 
        \item there exists a $1$-dimensional representation $\epsilon_{\p}$ of $\Delta_{\p}$ that occurs with multiplicity $1$ in $\rho_{|\Delta_{\p}}$. 
    \end{enumerate}
\end{ass}
 In the special case $d=2$ and $d^+=1$, such representations are expected to arise from Hecke eigenforms of weight $1$ on $\Gamma_1(N)$, where $N$ is the Artin conductor of $\rho$. The conjecture has been settled in various special cases, cf. \cite{buzzard2001icosahedral, taylor2003icosahedral} and references therein. The choice of the character $\epsilon_{\p}$ plays a role in the definition of the Selmer groups associated to $\rho$. Take $\mathcal{K}$ be the completion of $K$ at $\p$. There is a natural isomorphism $\op{Gal}(\mathcal{K}/\Q_p)\simeq \Delta_\p$, and set $\epsilon$ to denote the composite 
\[\op{Gal}(\mathcal{K}/\Q_p)\xrightarrow{\sim} \Delta_\p\xrightarrow{\epsilon_\p} \bar{\Q}^\times.\]Let $\chi$ denote the character associated to $\rho$ and $\Q(\chi)$ denote the field generated by the values of $\chi$. Let $\mathcal{F}$ be the field generated by $\Q_p$, the values of $\chi$ and the values of $\epsilon_{\p}$. We regard $\rho$ as a representation a vector space $V$ defined over $\cF$ and note that $\dim_{\cF}V=d$. Let $V^{\epsilon_\p}$ be the maximal $\cF$-subspace of $V$ on which $\Delta_\p$ acts by $\epsilon_\p$. By hypothesis, $\dim_\cF V^{\epsilon_\p}=1$. Let $\cO$ be the valuation ring in $\cF$ and $\varpi$ be its uniformizer. We assume that $\rho$ is irreducible, and therefore, there is a Galois stable $\cO$-lattice $T\subset V$. This lattice is uniquely determined up to scaling by a constant. We consider the divisible Galois module $D:=V/T$ and let $D^{\epsilon_\p}$ be the image of $V^{\epsilon_\p}$ in $D$. Let $\Q_{\infty}$ denote the cyclotomic $\Z_p$-extension of $\Q$. Let $\Q_n\subset \Q_\infty$ be the subextension such that $\op{Gal}(\Q_n/\Q)\simeq \Z/p^n\Z$. We set 
\[S_{\chi, \epsilon} (\Q_n):=\ker\left\{H^1(\Q_n, D)\longrightarrow \prod_{v\nmid p} H^1(\Q_{n, v}^{\op{nr}}, D)\times H^1(\Q_{n, \eta_p}^{\op{nr}}, D/D^{\epsilon_\p})\right\},\] where $\eta_p$ is the unique prime that lies above $p$.
\begin{theorem}[Greenberg and Vatsal]\label{GV finiteness}
    Suppose that the Assumption \ref{main ass} holds. Then, the Selmer group $S_{\chi, \epsilon}(\Q_n)$ is finite for all $n\in \Z_{\geq 0}$.
\end{theorem}
\begin{proof}
    The above result is \cite[Proposition 3.1]{greenbergvatsalArtinrepns}.
\end{proof}

\begin{definition}
    Define the Selmer group over $\Q_\infty$ to denote the direct limit with respect to restriction maps
    \[S_{\chi, \epsilon}(\Q_\infty):=\varinjlim_n S_{\chi, \epsilon} (\Q_n).\]
\end{definition}

Greenberg and Vatsal show that the Selmer group over $\Q_\infty$ is cofinitely generated and cotorsion over the Iwasawa algebra, which is a formal power series ring over $\cO$ in $1$-variable. Leveraging the results of Greenberg and Vatsal, we study the Euler characteristic formula associated to these Selmer groups and utilize it to study explicit conditions for the vanishing of Selmer group $S_{\chi, \epsilon}(\Q_\infty)$. For instance, we are able to prove that there is an explicit relationship between the vanishing of the Selmer group and the \emph{$p$-rationality} of $K$ (in the sense of \cite{movahhedi1990arithmetique}).
\begin{theorem}[Theorem \ref{cor 3.10}]
 Assume that  
     \begin{enumerate}[label=(\roman*)]
        \item the conditions of Assumption \ref{main ass} are satisfied.
        \item $H^0(\Delta_{\p}, D/D^{\epsilon_\p})=0$,
        \item $K$ is $p$-rational,
        \item $p$ does not divide the class number of $K$.
    \end{enumerate}
 Then, $S_{\chi, \epsilon}(\Q_\infty)=0$.
\end{theorem}

These criteria are illustrated in two special cases, namely that of $2$-dimensional irreducible Artin representations (cf. Theorem \ref{thm 4.12}) and $3$-dimensional icosahedral Artin representations of $A_5$ (cf. Theorem \ref{example thm 2}). The Theorem \ref{conditions GV} gives a more refined criterion that is equivalent to the vanishing of $S_{\chi, \epsilon}(\Q_\infty)$. The relationship between $p$-rationality and the Iwasawa invariants of number fields has been studied by Hajir and Maire, cf. \cite{hajirmaire}.  

\par For simplicity, we then specialize our discussion to $2$-dimensional Artin representation of dihedral type. Let $L$ be an imaginary quadratic field and $\zeta:\op{G}_L\rightarrow \bar{\Q}^\times$ is a character and $\rho=\op{Ind}_L^{\Q}\zeta:\op{G}_\Q\rightarrow \op{GL}_2(\bar{\Q})$ be the associated Artin representation. Set $\zeta'$ to be conjugate character defined by setting $\zeta'(x)=\zeta(c x c^{-1})$, where $c$ denotes the complex conjugation. Assume that $\zeta'=\zeta^{-1}$, thus $\rho$ is dihedral type. Let $S(\rho)$ be the set of primes $p$ such that 
\begin{enumerate}
    \item $p$ is odd and $p\nmid [K:\Q]$, 
    \item $p$ splits in $L$, $p\cO_L=\pi \pi^*$, and $\pi$ is inert in $K/L$.
\end{enumerate}
It is easy to see that if $\p|p$ is a prime of $K$, then, the conditions of Assumption \ref{main ass} are satisfied. Let $T(\rho)$ be the set of primes $p\in S(\rho)$ such that the Selmer group $S_{\chi, \epsilon}(\Q_\infty)$ vanishes for all of the primes $\p$ that lie above $p$. Set $T'(\rho):=S(\rho)\backslash T(\rho)$. In section \ref{s 4}, we apply our results to study the following natural question.
\begin{question}
    What can be said about the densities of the sets of primes in $T(\rho)$ and $T'(\rho)$?
\end{question}
A conjecture of Gras predicts that any number field $K$ is $p$-rational at all but finitely many primes $p$. This has conjectural implications to the vanishing of Selmer groups for all but finitely many primes for a compatible family of Artin representations. In the case when $K/\Q$ is an $S_3$-extension, we prove certain unconditional results by leveraging a result of Maire and Rougnant \cite{maire2020note} on the $p$-rationality of $S_3$-extensions of $\Q$.

\begin{theorem}[Theorem \ref{last theorem}]
    Let $K/\Q$ be an imaginary $S_3$ extension and $\rho$ be a $2$-dimensional Artin representation that factors through $\op{Gal}(K/\Q)$. Then,
    \[\# \{p\leq x\mid p\in T(\rho)\}\geq c\log x.\]
\end{theorem}
The result above is certainly weaker than what one is led to conjecture, however, it does well to illustrate the effectiveness of Theorem \ref{cor 3.10}. Throughout, we motivate our results by drawing upon analogues from the classical Iwasawa theory of class groups. 

\subsection{Organization}
Including the introduction, the article consists of four sections. In section \ref{s 2} we review the classical Iwasawa theory of class groups and Artin representations. We review results of Federer and Gross \cite{federer1980regulators} which gives an explicit relationship between $p$-adic regulators and Iwasawa invariants. We end this section by reviewing some results of Greenberg and Vatsal \cite{greenbergvatsalArtinrepns} on the Iwasawa theory of Artin representations. In the section \ref{s 3}, we discuss the notion of the Euler characteristic of a cofinitely generated and cotorsion module over the Iwasawa algebra. In section \ref{s 4}, we formulate and study some natural distribution questions, which serve to illustrate our results.

\section*{Statements and Declarations}

\subsection*{Conflict of interest} The authors report there are no conflict of interest to declare.

\subsection*{Data Availability} There is no data associated to the results of this manuscript. 

\subsection*{Acknowledgement} The authors thank the referees for helpful comments and suggestions.

\section{Iwasawa theory of class groups and Artin representations}\label{s 2}
\subsection{Classical Iwasawa theory and Iwasawa invariants}
\par We review the classical Iwasawa theory over the cyclotomic $\Z_p$-extension of a number field. The reader may refer to \cite{washington1997introduction} for a more comprehensive treatment. Throughout this section, we shall set $K$ to denote a number field and $p$ an odd prime number. Fix an algebraic closure $\bar{K}$ of $K$. All algebraic extensions considered in this article shall implicitly be assumed to be in $\bar{K}$. Let $\op{Cl}(K)$ denote the class group of $K$ and $\op{Cl}_p(K)$ its $p$-primary part. \par For $n\geq 1$, let $K(\mu_{p^n})$ be the number field extension of $K$ obtained by adjoining the $p^n$-th roots of unity to $K$. Denote by $K(\mu_{p^\infty})$ the union of all extensions $K(\mu_{p^n})$. There is a unique $\Z_p$-extension $K_\infty/K$ which is contained in $K(\mu_{p^\infty})$, which is referred to as the \emph{cyclotomic $\Z_p$-extension} of $K$. For $n\geq 1$, set $K_n/K$ to be the sub-extension of $K_\infty/K$, for which $\op{Gal}(K_n/K)$ is isomorphic to $\Z/p^n \Z$. Setting $K_0:=K$, refer to $K_n$ as the \emph{$n$-th layer}. Taking $\Gamma_n:=\op{Gal}(K_\infty/K_n)$ we identify $\op{Gal}(K_n/K)$ with $\Gamma/\Gamma_n$. For $n\in \Z_{\geq 0}$, denote by $H_p(K_n)$ the $p$-Hilbert class field of $K_n$. In other words, $H_p(K_n)$ is the maximal unramified abelian $p$-extension of $K_n$ in $\bar{K}$. Set $X_n$ to denote the Galois group $\op{Gal}(H_p(K_n)/K_n)$, and identify $X_n$ with the maximal $p$-primary quotient of $\op{Cl}(K_n)$. Thus, $[H_p(K_n):K_n]$ is equal to $\# \op{Cl}_p(K_n)$. Since the primes above $p$ are totally ramified in $K_\infty$, it follows that $H_p(K_n)\cap K_\infty=K_n$, and thus, there are surjective maps $X_m\rightarrow X_n$ for all $m\geq n$. Taking $X_\infty$ to be the inverse limit $\varprojlim_n X_n$, we find that $X_\infty$ is both a $\Z_p$-module as well as a module over $\Gamma$. On the other hand, letting $H_p(K_\infty)$ denote the maximal unramified abelian pro-$p$ extension of $K_\infty$, we identify $X_\infty$ with the Galois group $\op{Gal}(H_p(K_\infty)/K_\infty)$.

\par On order to better study the algebraic structure of $X_\infty$, it proves fruitful to view $X_\infty$ as a module over a completed group ring known as the Iwasawa algebra. This algebra is defined as follows
\[\Lambda:=\Z_p\llbracket \Gamma \rrbracket =\varprojlim_{n} \Z_p[\Gamma /\Gamma_n].\]
Choosing a topological generator $\gamma\in\Gamma$, we identify $\Lambda$ with the formal power series ring $\Z_p\llbracket T\rrbracket $, by setting $T:=(\gamma-1)$. As a $\Lambda$-module, $X_\infty$ is finitely generated and torsion \cite[chapter 13]{washington1997introduction}. 
\par Let $\cO$ be a valuation ring with residue characteristic $p$, and let $\varpi$ be a uniformizer of $\cO$. Then, the Iwasawa algebra over $\cO$ is defined by extending coefficients to $\cO$, as follows $\Lambda_{\cO}:=\Lambda \otimes_{\Z_p}\cO$. The Iwasawa algebra $\Lambda_\cO$ is a local ring with maximal ideal $\mathfrak{m}=(\varpi, T)$. A polynomial $f(T)\in \cO\llbracket T\rrbracket$ is said to be \emph{distinguished} if it is a monic polynomial whose non-leading coefficients are divisible by $\varpi$. The Weierstrass preparation theorem states that any power series $f(T)$ decomposes into a product
\[f(T)=\varpi^\mu \times g(T)\times u(T),\]where $\mu\in \Z_{\geq 0}$, $g(T)$ is a distinguished polynomial and $u(T)$ is a unit in $\Lambda_\cO$. The $\mu$-invariant of $f(T)$ is the power of $\varpi$ above, and the $\lambda$-invariant is the degree of the distinguished polynomial $g(T)$. The prime ideals of height $1$ are the principal ideals $(\varpi)$ and $(g(T))$, where $g(T)$ is an irreducible distinguished polynomial. 

\par Given a finitely generated and torsion $\Lambda_{\cO}$-module $M$, there is a homomorphism of $\Lambda_{\cO}$-modules
\begin{equation}\label{structural homomorphism}M \longrightarrow  \left(\bigoplus_{i=1}^s\frac{\cO\llbracket T\rrbracket}{(\varpi^{\mu_i})}\right) \oplus \left(\bigoplus_{j=1}^t \frac{\cO\llbracket T\rrbracket}{(f_i(T)^{\lambda_i})}\right),\end{equation}
with finite kernel and cokernel.
Here, $\mu_i, \lambda_j\geq 0$, and $f_j(T)$ are irreducible distinguished polynomials. For further details, we refer to \cite[Theorem 3.12]{washington1997introduction}.

\begin{definition}\label{def of mu and lambda}
    The $\mu$-invariant $\mu_p(M)$ is the sum of the entries $\sum_{i=1}^s \mu_i$ if $s>0$, and is set to be $0$ if $s=0$. On the other hand, the $\lambda$-invariant $\lambda_p(M)$ is defined to be $\sum_{i=1}^s \lambda_i \op{deg}(f_i)$ if $s>0$, and defined to be $0$ if $t=0$.
    The characteristic element $f_M(T)$ is the product 
    \[f_M(T):=\prod_i \varpi^{\mu_i} \times \prod_j f_j(T)^{\lambda_j}.\]
\end{definition} We remark that the $\mu$-invariant and $\lambda$-invariant of $f_M(T)$ are $\mu_p(M)$ and $\lambda_p(M)$ respectively.
\begin{proposition}
    Suppose that $M$ is a finitely generated and torsion $\Lambda_{\cO}$-module. Then, the following assertions hold,
    \begin{enumerate}
        \item $\mu_p(M)=0$ if and only if $M$ is finitely generated as an $\cO$-module. In this case, $\lambda_p(M)$ is the $\cO$-rank of $M$.
        \item Letting $r_p(M)$ denote the order of vanishing of $f_M$ at $0$, we have that \[\lambda_p(M)\geq r_p(M).\] 
        \item Write $f_M(T)$ as a power series 
        \[f_M(T)=a_r T^r+a_{r+1} T^{r+1}+\dots+ a_\lambda T^{\lambda},\] where $r=r_p(M)$ and $\lambda=\lambda_p(M)$. The $\mu$-invariant $\mu_p(M)=0$ if and only if there is a coefficient $a_i$ not divisible by $\varpi$.
        \item We have that $\varpi^\mu|| a_\lambda$ and that $\varpi^{\mu+1}|a_i$ for all $i<\lambda$.
    \end{enumerate}
    \end{proposition}
    \begin{proof}
        The results are easy observations that follow from the structural homomorphism \eqref{structural homomorphism} and the definition of the Iwasawa invariants. 
    \end{proof}

    \begin{remark}
        Let $\cO'$ be a valuation ring that is a finite extension of $\cO$, and $e$ be its ramification index. Set $M_{\cO'}:=M\otimes_{\cO} \cO'$ and regard $M_{\cO'}$ as a module over $\Lambda_{\cO'}$. Then, it is easy to see that \[\mu_p(M_{\cO'})=e\mu_p(M)\text{ and }\lambda_p(M_{\cO'})=\lambda_p(M).\]
    \end{remark}
 \par We denote the $\mu$-invariant (resp. $\lambda$-invariant) of $X_\infty$ by $\mu_p(K)$ (resp. $\lambda_p(K)$). In this setting, $\cO:=\Z_p$. Iwasawa proved that for all large enough values of $n$, there is an invariant $\nu_p(K)\in \Z$ for which 
\[\log_p\left(\# \op{Cl}_p(K_n)\right)=p^n \mu_p(K)+n \lambda_p(K)+\nu_p(K),\] cf. \cite[Theorem 13.13]{washington1997introduction}. Moreover, Iwasawa conjectured that $\mu_p(K)=0$ for all number fields $K$, cf. \cite{iwasawa1973z}. For abelian extensions $K/\Q$ the conjecture has been proven by Ferrero and Washington \cite{ferrero1979iwasawa}.

\subsection{The leading coefficient of the characteristic series}

\par Let $K$ be a CM field with totally real subfield $K^+$. The Galois group $\op{Gal}(K/K^+)$ acts on $X_\infty$. Let $\tau$ be the generator of $\op{Gal}(K/K^+)$ and set
\[X_\infty^-:=\{x\in X_\infty\mid \tau(x)=-x\}. \] Then, $X_\infty^-$ is a $\Lambda$-module whose $\mu$-invariant (resp. $\lambda$-invariant) is denoted $\mu_p^-(K)$ (resp. $\lambda_p^-(K)$). Let $f_{K}^-(T)$ be the characteristic element associated to $X_\infty^-$. Take $I$ to be the set of primes of $K^+$ that lie above $p$ and split in $K$, and set $r_{p,K}:=\# I$. Let $S$ be the set of places of $K$ dividing $p$ and $\infty$. Let $U=U_K$ be the group of $S$-units of $K^*$ and let $M=M_K$ be the free abelian group of divisors at $S$. Given a $\op{Gal}(K/K^+)$-module $N$, let 
\[N^-:=\{n\in N\mid \tau(n)=-n\}.\] Let $R$ be a ring, then we set $R N^-$ to denote the extension of scalars $N^-\otimes R$. 
It is easy to see that both $U^-$ and $M^-$ are free abelian group of rank $r_{p, K}$. 
\par The map 
\[\phi:U^-\rightarrow M^-\] is defined by setting 
\[\phi(x):=\sum_{v|p} \op{ord}_p\left(\op{Norm}_{K_v/\Q_p}x \right)v\]
The induced map 
\[\phi: \Q U^-\rightarrow \Q M^-\] is an isomorphism (cf. \cite[Proposition 1.4]{federer1980regulators}). The inverse of $\phi$ can be described as follows. For each prime $v\in I$, we choose a prime $\widetilde{v}|v$ of $K$. Note that $\Q M^-$ has a basis \[\mathcal{B}=\{\widetilde{v}-\tau(\widetilde{v})\mid v\in I\}.\] Let $\p$ be the prime ideal associated to $v$, and $h$ be a positive integer such that $\p^h= (\alpha)$ is principal. Both $\alpha$ and $\tau(\alpha)$ are elements of $U$, and $\alpha/\tau(\alpha)\in U^-$. Setting $f_v$ to denote the residue class degree of $v$, take 
\[\phi^{-1}\left(\widetilde{v}-\tau(\widetilde{v})\right):=\frac{1}{h f_v}\otimes \left(\alpha/\tau(\alpha)\right).\] Then, $\phi^{-1}$ is the inverse to $\phi$. 

\par Define the homomorphism 
\[\lambda: U^-\rightarrow \Q_p M^-\] by setting 
\[\lambda(y):=\sum_{v|p} 
\log_p\left(N_{K_v/\Q_p}(y)\right).\] Composing $\phi^{-1}$ with $\lambda$, we obtain an endomorphism 
\[\lambda \phi^{-1}: \Q_p M^-\rightarrow \Q_p M^-.\]
\begin{definition}
    With notation as above, define the regulator $\op{Reg}_p(K)$ as follows
    \[\op{Reg}_p(K):=\det \left(\lambda\phi^{-1}\mid \Q_p M^{-}\right),\] where it is understood that $\op{Reg}_p(K):=1$ when $r_{p,K}=0$. 
\end{definition}

\begin{theorem}[Iwasawa, Greenberg]\label{order of vanishing}
    With respect to the notation above, $f_K^-(T)$ is divisible by $T^{r_{p, K}}$. 
\end{theorem}
\begin{proof}
    We refer to works of Iwasawa \cite{iwasawa1973z} and Greenberg \cite{greenberg1973certain} for the proof of the statement.
\end{proof}
As a consequence of Theorem \ref{order of vanishing}, one may write 
\[f_K^-(T)=a_r T^r+a_{r+1} T^{r+2}+\dots+ a_\lambda T^\lambda, \] where $r=r_{p, K}$ and $\lambda:=\lambda_p^-(K)$. We note that as a consequence, 
\[ \lambda_p^-(K)\geq r_{p,K}. \]For $a,b\in \Q_p$, we write $a\sim b$ to mean that $a=ub$ for some $u\in \Z_p^\times$. Let $h_K$ (resp. $h_K^-$) denote the class number (resp. number of elements in the minus part of the class group) of $K$. Let $w_{K(\mu_p)}$ be the number of roots of unity contained in $K(\mu_p)$. 
\begin{theorem}[Federer-Gross]
    With respect to the notation above, the following are equivalent
    \begin{enumerate}
        \item $\op{Reg}_p(K)\neq 0$,
        \item $a_r\neq 0$.
    \end{enumerate}
    Assuming that these equivalent conditions hold, we have that 
    \[a_r\sim \frac{h_K^- \left(\prod_{v\in I} f_v\right) \op{Reg}_p(K)}{w_{K(\mu_p)}^{r_{p,K}}}.\]
\end{theorem}
\begin{proof}
    The above result is \cite[Proposition 3.9]{federer1980regulators}.
\end{proof}
Suppose that $p\nmid f_v$ for all $v\in I$, and $p\nmid h_K$. Then, the map $\phi^{-1}$ is defined over $\Z_p$ and it is easy to see that $\op{Reg}_p(K)$ is divisible by $p^{r_{p,K}}$. The normalized regulator is defined as follows
\[\mathcal{R}_p(K):=\frac{\op{Reg}_p(K)}{p^{r_{p,K}}} .\]
\begin{corollary}\label{cor2.5}
    Let $K$ be a CM field for which 
    \begin{enumerate}[label=(\roman*)]
        \item $\mathcal{R}_p(K)\neq 0$,
        \item $w_{K(\mu_p)}\sim p$,
        \item $p\nmid f_v$ for all $v\in I$,
        \item $p\nmid h_K$.
    \end{enumerate}
    Then, the following are equivalent
    \begin{enumerate}
        \item\label{c1cor2.5} $\mu_p^-(K)=0$ and $\lambda_p^-(K)=r_{p,K}$,
        \item\label{c2cor2.5} $a_r$ is a unit in $\Z_p$,
        \item\label{c3cor2.5} $p\nmid \mathcal{R}_p(K)$.
    \end{enumerate}
\end{corollary}
\begin{proof}
    With respect to above notation, write $f_K^-(T)=T^{r_{p,K}} g_K^-(T)$, where $a_r=g_K^-(0)$. Then, $a_r$ is a unit in $\Z_p$ if and only if $g_K^-(T)$ is a unit in $\Lambda$. This implies that
    \[\mu_p^-(K)=0\text{ and }\lambda_p^-(K)=r_{p,K}.\] Conversely, if \[\mu_p^-(K)=0\text{ and }\lambda_p^-(K)=r_{p,K},\] then the factorization $f_K^-(T)=T^{r_{p,K}} g_K^-(T)$ implies that $g_K^-(T)$ must be a unit in $\Lambda$. Thus, we find that \eqref{c1cor2.5} and \eqref{c2cor2.5} are equivalent. It follows from our assumptions that $a_r\sim \mathcal{R}_p(K)$. Therefore, $a_r$ is a unit in $\Z_p$ if and only if $p\nmid \mathcal{R}_p(K)$. This shows that the conditions \eqref{c2cor2.5} and \eqref{c3cor2.5} are equivalent. 
\end{proof}
When $K$ is an imaginary quadratic field, $X_\infty=X_\infty^-$ and thus $\mu_p(K)=\mu_p^-(K)$ and $\lambda_p(K)=\lambda_p^-(K)$. In this setting, we find that 
\[r_{p, K}:=\begin{cases}
    & 0\text{ if }p\text{ splits in }K;\\
    & 1\text{ if }p\text{ is inert or ramified in }K.\\
\end{cases}\]
Note that $\mu_p(K)=0$ by the aforementioned result of Ferrero and Washington \cite{ferrero1979iwasawa}. 
\begin{corollary}\label{cor2.6}
    Let $K$ be an imaginary quadratic field and $p$ be an odd prime number. Then, the following assertions hold.
    \begin{enumerate}
        \item Suppose that $p$ splits in $K$. Then, $\lambda_p(K)=0$ if and only if $p\nmid h_K$. 
        \item Suppose that $p$ is inert in $K$. Then, $\lambda_p(K)=1$ if and only if $p\nmid h_K$ and $p\nmid \mathcal{R}_p(K/\Q)$. 
    \end{enumerate}
\end{corollary}
\begin{proof}
    The result is a special case of Corollary \ref{cor2.5}.
\end{proof}

\par Let $K$ be an imaginary quadratic field and $p$ be an odd prime which is inert in $K$. In this case $\lambda_p(K)\geq 1$ and the Corollary \ref{cor2.6} asserts that $\lambda_p(K)=1$ if and only if $\mathcal{R}_p(K/\Q)$ is a unit in $\Z_p$. The analysis of this condition leads to the statement of Gold's criterion.

\begin{theorem}[Gold's criterion]
    Let $K$ be an imaginary quadratic field and $p$ be an odd prime number which splits in $K$. Assume that $p\nmid h_K$. Let $\p|p$ be a prime ideal, and $r$ be a positive integer not divisible by $p$, such that $\p^r$ is principal. Let $\alpha\in\cO_K$ be a generator of $\p^r$. Setting $\bar{\p}$ to be the complex conjugate of $\p$, the following conditions are equivalent
    \begin{enumerate}
        \item $\lambda_p(K)>1$,
        \item $\alpha^{p-1}\equiv 1\left(\mod{\bar{\p}^2}\right)$. 
    \end{enumerate}
\end{theorem}

\begin{proof}
    The result is a consequence of \cite[Theorems 3 and 4]{gold1974nontriviality}, and can also be seen to follow from Corollary \ref{cor2.6}, cf. \cite[proof of Proposition 2.1]{sands1993non}.
\end{proof}

\subsection{Artin representations}
\par We briefly discuss the results of Greenberg and Vatsal \cite{greenbergvatsalArtinrepns} on the Iwasawa theory of Selmer groups associated to Artin representations. Let $K/\Q$ be a finite Galois extension with Galois group $\Delta:=\op{Gal}(K/\Q)$. Fix an irreducible Artin representation of dimension $d>1$ \[\rho: \Delta\rightarrow \op{GL}_d(\bar{\Q})\] and \[S_{\chi, \epsilon}(\Q_\infty):=\varinjlim_n S_{\chi, \epsilon} (\Q_n)\] the Selmer group over the cyclotomic $\Z_p$-extension of $\Q$. Note that $S_{\chi, \epsilon}(\Q_n)$ (resp. $S_{\chi, \epsilon}(\Q_\infty)$) is an $\cO[\Gamma/\Gamma_n]$-module (resp. $\Lambda_{\cO}$-module). 

\par It is easy to see that the restriction map
\[H^1(\Q_n, D)\rightarrow H^1(\Q_\infty, D)^{\Gamma_n}\] induces a map between Selmer groups
\[\iota_n:S_{\chi, \epsilon}(\Q_n)\rightarrow S_{\chi, \epsilon}(\Q_\infty)^{\Gamma_n}.\] The following control theorem shows that this map is injective with cokernel which is independent of $n$.

\begin{theorem}[Greenberg and Vatsal -- Control theorem]\label{control theorem}
    Suppose that $\chi$ is nontrivial, then, $\iota_n$ fits into a short exact sequence of $\Gamma/\Gamma_n$-modules
    \[0\rightarrow S_{\chi, \epsilon}(\Q_n)\rightarrow S_{\chi, \epsilon}(\Q_\infty)^{\Gamma_n}\rightarrow H^0(\Delta_{\p}, D/D^{\epsilon_\p})\rightarrow 0.\]
\end{theorem}

\begin{proof}
The above result is \cite[Proposition 4.1]{greenbergvatsalArtinrepns}.
\end{proof}
We note that $H^0(\Delta_\p, D/D^{\epsilon_\p})\simeq (\cF/\cO)^t$, where $t$ is the multiplicity of the trivial representation of $\Delta_{\p}$ in $V/V^{\epsilon_\p}$. Here, the action of $\Gamma/\Gamma_n$ on $H^0(\Delta_\p, D/D^{\epsilon_\p})$ is trivial.

\par Let $\cD$ be a discrete $p$-primary $\Lambda_\cO$-module. We take $\cD^\vee$ to be the Pontryagin dual defined as follows
\[\cD^\vee:=\op{Hom}_{\Z_p} \left(\cD, \Q_p/\Z_p\right).\] Say that $\cD$ is cofinitely generated (resp. cotorsion) as a $\Lambda_\cO$-module, if $\cD^\vee$ is finitely generated (resp. torsion) as a $\Lambda_\cO$-module. We set $X_{\chi, \epsilon}(\Q_\infty)$ be the Pontryagin dual of $S_{\chi, \epsilon} (\Q_\infty)$.

\begin{theorem}[Greenberg-Vatsal]\label{no finite submodules}
    With respect to the notation above $X_{\chi, \epsilon}(\Q_\infty)$ is a finitely generated and torsion $\LcO$-module that contains no non-trivial finite $\LcO$-submodules.
\end{theorem}

\begin{proof}
    We refer to \cite[Propositions 4.5 and 4.7]{greenbergvatsalArtinrepns} for the proof of the above result.
\end{proof}

 Set $\mu_{\chi, \epsilon}$ and $\lambda_{\chi, \epsilon}$ to denote the $\mu$ and $\lambda$-invariants of $S_{\chi, \epsilon}(\Q_\infty)^\vee$ respectively. The characteristic series of $S_{\chi, \epsilon}(\Q_\infty)^\vee$ is denoted $f_{\chi, \epsilon}(T)$.

\section{The Euler characteristic}\label{s 3}

In this section, we introduce the notion of the Euler characteristic associated to a cofinitely generated and cotorsion module over $\Lambda_{\cO}$, and introduce conditions for it to be well defined. For the Selmer group associated to an Artin representation, the Euler characteristic is shown to equal the cardinality of $S_{\chi, \epsilon}(\Q)$. The results presented in this section have consequences to the study of the Iwasawa $\mu$ and $\lambda$-invariants associated to an Artin representation. 
\subsection{Definition and properties of the Euler characteristic}
\par Let $\cD$ be a cofinitely generated and cotorsion module over $\LcO$. Since $\Gamma$ is pro-cyclic, it has cohomological dimension $1$, we find that $H^i(\Gamma, \cD)=0$ for $i\geq 2$. Also note that $H^1(\Gamma, \cD)$ is identified with the module of coinvariants $\cD_\Gamma$. This invariant encodes information about the valuation of the constant term of the characteristic series. For an Artin representation that satisfies the conditions of Assumption \ref{main ass} we discuss conditions for the Euler characteristic to be well defined, and give an explicit formula for it. The Euler characteristic formula in this context can be viewed as an analogue of the result of Gross and Federer. 

\begin{proposition}\label{prop 4.1}
    Let $\cD$ be a cofinitely generated and cotorsion $\LcO$-module, then, 
    \begin{enumerate}
        \item\label{p1prop4.1} $\op{corank}_{\Z_p} H^0(\Gamma, \cD)=\op{corank}_{\Z_p} H^1(\Gamma, \cD)$.
        \item\label{p2prop4.1} The module $H^0(\Gamma, \cD)$ is finite if and only if $H^1(\Gamma, \cD)$ is finite.
    \end{enumerate}
   
\end{proposition}

\begin{proof}
    Since $\cD$ is cofinitely generated as a $\LcO$-module, it follows that both $\cD^\Gamma$ and $\cD_\Gamma$ are cofinitely generated as a $\Z_p$-module. Part \eqref{p1prop4.1} follows from \cite[Theorem 1.1]{howson2002euler}. Since $\cD$ is cofinitely generated as a $\Lambda$-module, it follows that $\cD^\Gamma$ and $\cD_\Gamma$ are both cofinitely generated as $\Z_p$-modules. Part \eqref{p2prop4.1} follows from \eqref{p1prop4.1}.
\end{proof}

\begin{definition}\label{def of euler char}
    Let $\cD$ be a cofinitely generated and cotorsion $\LcO$-module. Then, we say that the Euler characteristic of $\cD$ is well defined if $H^0(\Gamma, \cD)$ (or equivalently $H^1(\Gamma, \cD)$) is finite. Assuming that the characteristic of $\cD$ is well defined, the Euler characteristic is defined as follows
    \[\chi(\Gamma, \cD):=\prod_{i\geq 0} \# H^i(\Gamma, \cD)^{(-1)^i}=\frac{\# H^0(\Gamma, \cD)}{\# H^1(\Gamma,\cD)}.\]
\end{definition}

Recall from Definition \ref{def of mu and lambda} that $f_{\cD^\vee}(T)$ is the characteristic element of $\cD^\vee$ and that $\mu_p(\cD^\vee)$ (resp. $\lambda_p(\cD^\vee)$) is the $\mu$-invariant (resp. $\lambda$-invariant). For the ease of notation, we set
\[f_\cD(T):=f_{\cD^\vee}(T), \mu_p(\cD):=\mu_p(\cD^\vee)\text{ and }\lambda_p(\cD):=\lambda_p(\cD^\vee). \]From the expansion
\[f_\cD(T)=\sum_{i=0}^\lambda a_i T^i,\] we note that $\lambda=\lambda_p(\cD)$. Moreover, $\mu_p(\cD)=0$ if and only if $\varpi \nmid a_i$ for at least one coefficient $a_i$. In particular, we have the have that $\varpi\nmid f_\cD(0)$ if and only if $\mu_p(\cD)=\lambda_p(\cD)=0$. Fix an absolute value  $|\cdot |_\varpi$ on $\cO$ normalized by setting 
\[|\varpi|_\varpi:=[\cO:\varpi\cO]^{-1}.\]

\begin{proposition}\label{prop4.3}
    Let $\cD$ be a cofinitely generated and cotorsion $\Lambda$-module. Then, with respect to the above notation, the following conditions are equivalent
    \begin{enumerate}
        \item $f_\cD(0)\neq 0$, 
        \item the Euler characteristic $\chi(\Gamma, \cD)$ is well defined.
    \end{enumerate}
    Moreover, if the above equivalent conditions are satisfied, then, 
    \[\chi(\Gamma, \cD)= |f_\cD(0)|_\varpi^{-1}.\]
\end{proposition}
\begin{proof}
    It is easy to see that if $\cD$ and $\cD'$ are cofinitely generated $\LcO$-modules that are pseudo-isomorphic, then $\cD^\Gamma$ is finite if and only if $(\cD')^\Gamma$ is finite. Therefore, $\chi(\Gamma, \cD)$ is well defined if and only if $\chi(\Gamma, \cD')$ is well defined. The characteristic series is determined up to pseudo-isomorphism, and therefore, $f_\cD(T)=f_{\cD'}(T)$. We assume without loss of generality that
    \[\cD^\vee=\left(\bigoplus_{i=1}^s\frac{\cO\llbracket T\rrbracket}{(\varpi^{\mu_i})}\right) \oplus \left(\bigoplus_{j=1}^t \frac{\cO\llbracket T\rrbracket}{(f_i(T)^{\lambda_i})}\right).\] We identify $\left(\cD^\Gamma\right)^\vee$ with \[\left(\cD^\vee\right)_\Gamma=\cD^\vee/T\cD^\vee \simeq \left(\bigoplus_{i=1}^s\cO/(\varpi^{\mu_i})\right) \oplus \left(\bigoplus_{j=1}^t \cO/(f_i(0)^{\lambda_i})\right).\]
    Since
    \[f_\cD(0)=\prod_i \varpi^{\mu_i}\times \prod_j f_j(0)^{\lambda_j},\]we deduce that \[\# H^0(\Gamma, \cD)=\# H^1(\Gamma, \cD^\vee)= \# \left(\cD^\vee\right)_\Gamma\] is finite if and only if $f_\cD(0)$ is non-zero. Therefore, the Euler characteristic is well defined if and only if $f_\cD(0)$ is non-zero.
    
    \par Assuming that $f_\cD(0)\neq 0$, we calculate the Euler characteristic. First, we note that the previous argument implies that
    \[\# H^0(\Gamma, \cD)= |f_\cD(0)|_\varpi^{-1}.\] On the other hand, we find that
    \[\# H^1(\Gamma, D)=\# H^0(\Gamma, D^\vee)=\# \left(D^\vee\right)^\Gamma.\]Identify $(\cD^\vee)^\Gamma$ with the kernel of the multiplication by $T$ endomorphism of $\cD^\vee$. Since $f_{\cD}(0)$ is assumed to be non-zero, it follows that none of the terms $f_j(T)$ are divisible by $T$. It follows from this that the multiplication by $T$ map is injective and hence, $(\cD^\vee)^\Gamma=0$. Thus it has been shown that 
    \[\chi(\Gamma, D)=|f_\cD(0)|_\varpi^{-1}.\]
    \end{proof}

    \begin{corollary}\label{cor4.4}
        Suppose that the Euler characteristic of $\cD$ is well defined. Then the following conditions are equivalent
        \begin{enumerate}
            \item\label{p1cor4.4} $\chi(\Gamma, \cD)=1$,
            \item\label{p2cor4.4} $\mu_p(\cD)=0$ and $\lambda_p(\cD)=0$.
        \end{enumerate}
    \end{corollary}
    \begin{proof}
        It is easy to see that $\mu_p(\cD)=0$ and $\lambda_p(\cD)=0$ if and only if $f_\cD(T)$ is a unit in $\Lambda_\cO$. Thus, the condition \eqref{p2cor4.4} is equivalent to the condition that $\varpi\nmid f_\cD(0)$. According to Proposition \ref{prop4.3}, 
        \[\chi(\Gamma, \cD)=|f_\cD(0)|_\varpi^{-1}.\] Therefore $\chi(\Gamma, \cD)=1$ if and only if $f_\cD(0)$ is a unit in $\cO$. 
    \end{proof}

    \subsection{Calculating the Euler characteristic}
Let $\rho$ be an Artin representation with character $\chi$ and set $\epsilon$ to be the chosen character with respect to which the Selmer group $S_{\chi, \epsilon}(\Q_\infty)$ is defined. We shall assume throughout that the Assumption \ref{main ass} is satisfied. 

\begin{proposition}\label{well defined EC}
    With respect to above notation, the following conditions are equivalent
    \begin{enumerate}
        \item the Euler characteristic $\chi(\Gamma, S_{\chi, \epsilon}(\Q_\infty))$ is defined, 
        \item $H^0(\Delta_{\p}, D/D^{\epsilon_\p})=0$.
    \end{enumerate}
\end{proposition}
\begin{proof}
    The Euler characteristic is defined if and only if $S_{\chi, \epsilon}(\Q_\infty)^{\Gamma}$ is finite. By the control theorem, we have the short exact sequence
    \[0\rightarrow S_{\chi, \epsilon}(\Q)\rightarrow S_{\chi, \epsilon}(\Q_\infty)^{\Gamma}\rightarrow H^0(\Delta_{\p}, D/D^{\epsilon_\p})\rightarrow 0,\] where $H^0(\Delta_{\p}, D/D^{\epsilon_\p})$ is a cofree $\cO$-module. According to Theorem \ref{GV finiteness}, the Selmer group $S_{\chi, \epsilon}(\Q)$ is finite. Therefore, the Euler characteristic is well defined if and only if $H^0(\Delta_{\p}, D/D^{\epsilon_\p})=0$.
\end{proof}

\begin{theorem}\label{Euler char formula}
    Let $\rho$ be an Artin representation for which 
    \begin{enumerate}
        \item the conditions of Assumption \ref{main ass} are satisfied.
        \item $H^0(\Delta_{\p}, D/D^{\epsilon_\p})=0$.
    \end{enumerate}
    Then, the Euler characteristic is given by 
    \[\chi(\Gamma, S_{\chi, \epsilon}(\Q_\infty))=\# S_{\chi, \epsilon}(\Q).\]
\end{theorem}

\begin{proof}
    According to Proposition \ref{well defined EC}, the Euler characteristic is well defined. We are to calculate the orders of the finite abelian groups $S_{\chi, \epsilon}(\Q_\infty)^{\Gamma}$ and $S_{\chi, \epsilon}(\Q_\infty)_{\Gamma}$. According to Theorem \ref{no finite submodules}, $X_{\chi, \epsilon}(\Q_\infty):=S_{\chi, \epsilon}(\Q_\infty)^\vee$ does not contain any nontrivial finite $\Lambda_{\cO}$-submodules. Since the Euler characteristic is well defined, $X_{\chi, \epsilon}(\Q_\infty)^{\Gamma}$ is finite, and hence is zero. In other words, 
    $S_{\chi, \epsilon}(\Q_\infty)_{\Gamma}=0$. On the other hand, it follows from Theorem \ref{control theorem} that 
    \[S_{\chi, \epsilon}(\Q_\infty)^{\Gamma}=S_{\chi, \epsilon}(\Q).\]Therefore, we find that $\chi(\Gamma, S_{\chi, \epsilon}(\Q_\infty))=\# S_{\chi, \epsilon}(\Q)$.
\end{proof}

\begin{corollary}\label{cor 4.7}
    Let $\rho$ be an Artin representation and $\epsilon:\op{Gal}(\cK/\Q_p)\rightarrow \cO^\times$ be a character for which the following assumptions are satisfied.
     \begin{enumerate}[label=(\roman*)]
        \item The conditions of Assumption \ref{main ass} are satisfied.
        \item $H^0(\Delta_{\p}, D/D^{\epsilon_\p})=0$.
    \end{enumerate}
    Then, the following conditions are equivalent
    \begin{enumerate}
        \item\label{c1cor4.7} $S_{\chi, \epsilon}(\Q_\infty)=0$,
        \item\label{c2cor4.7} $S_{\chi, \epsilon}(\Q_\infty)$ is finite,
        \item\label{c3cor4.7} $\mu_{\chi, \epsilon}=\lambda_{\chi, \epsilon}=0$, 
        \item\label{c4cor4.7} $\chi(\Gamma, S_{\chi, \epsilon}(\Q_\infty))=1$,
        \item\label{c5cor4.7} $S_{\chi, \epsilon}(\Q)=0$.
    \end{enumerate}
\end{corollary}

\begin{proof}
\par It is clear that \eqref{c1cor4.7} implies \eqref{c2cor4.7}. On the other hand, by Theorem \ref{no finite submodules}, $X_{\chi, \epsilon}(\Q_\infty)$ contains no nontrivial finite $\Lambda_\cO$-submodules. Hence, if $S_{\chi, \epsilon}(\Q_\infty)$ is finite, it must be equal to $0$. Therefore, \eqref{c1cor4.7} and \eqref{c2cor4.7} are equivalent.
\par We note that a $\Lambda_\cO$-module $M$ is finite if and only if it is pseudo-isomorphic to the trivial module. On the other hand, $M\sim 0$ if and only if the $\mu$ and $\lambda$-invariants of $M$ are $0$. Therefore \eqref{c2cor4.7} and \eqref{c3cor4.7} are equivalent.
\par The equivalence of \eqref{c3cor4.7} and \eqref{c4cor4.7} follows from Corollary \ref{cor4.4}, and the equivalence of \eqref{c4cor4.7} and \eqref{c5cor4.7} follows from Proposition \ref{prop4.3}.
\end{proof}

\subsection{The structure of the Selmer group $S_{\chi, \epsilon}(\Q)$}
\par In this subsection, we analyze the structure of $S_{\chi, \epsilon}(\Q)$ and establish conditions for the vanishing of $S_{\chi, \epsilon}(\Q_\infty)$. We set $d(\chi)$ to denote $d$ and write $d(\chi)=d^+(\chi)+d^-(\chi)$, reflecting the action of the archimedian places on the representation. At each prime $\p|p$, let $\cU_\p$ denote the principal units at $\p$, and set $\cU_p$ to denote the product $\prod_{\p|p} \cU_\p$. The decomposition group $\Delta_{\p}$ naturally acts on $\cU_{\p}$, and $\cU_p$ is identified with the induced representation $\op{Ind}_{\Delta_\p}^\Delta \cU_{\p}$. Let $U_K$ be the group of units of $\cO_K$ that are principal units at all primes $\p|p$. The diagonal inclusion map 
\[U_K\hookrightarrow \cU_p\] is a $\Delta$-equivariant map, and induces a $\Z_p$-linear map 
\[\lambda_p: U_K\otimes \Z_p\rightarrow \cU_p.\] Let $\bar{U}_{p}$ denote the closure of the image of the map $\lambda_p$. We recall that the character $\epsilon$ arises from a choice of prime $\p|p$ and character $\epsilon_\p: \Delta_\p\rightarrow \cO^\times$. Let $\p'$ be any other prime above $p$, and write $\p'=\delta(\p)$ for $\delta\in \Delta$. Then conjugation by $\delta$ gives an isomorphism $c_\delta:\Delta_{\p'}\xrightarrow{\sim} \Delta_\p$. We set $\epsilon_{\p'}$ to denote the composite 
\[\Delta_{\p'}\xrightarrow{c_\delta} \Delta_\p\xrightarrow{\epsilon_\p}\cO^\times.\]Thus for each prime $\p|p$, we have made a choice of character $\epsilon_p$ of $\Delta_\p$. Given a $\Z_p$-module $M$, set $M_{\cO}:=M\otimes_{\Z_p}\cO$. For $\p|p$, we have a decomposition of $\cO[\Delta_\p]$-modules, 
\[\cU_{\p, \cO}=\cU_{\p, \cO}^{\epsilon_\p}\times \mathcal{V}_{\p, \cO},\]where the action of $\Delta_\p$ on $\cU_{\p, \cO}^{\epsilon_\p}$ is via $\epsilon_\p$.  
Use the following notation
\[\cU_{p, \cO}^{[\epsilon]}:=\prod_{\p|p} \cU_{\p, \cO}^{\epsilon_\p}, \text{ and } \mathcal{V}_{p, \cO}=\prod_{\p|p} \mathcal{V}_{\p, \cO}.\] Then, we find that 
\[\cU_{p, \cO}=\cU_{p, \cO}^{[\epsilon]}\times \mathcal{V}_{p, \cO}.\]
The above decomposition is that of $\cO[\Delta]$-modules. In the same way as above, $\bar{U}_{p,\cO}$ decomposes into a product 
\[\bar{U}_{p, \cO}=\bar{U}_{p, \cO}^{[\epsilon]}\times \bar{V}_{p, \cO}.\]

Let $\left(\bar{U}_{p, \cO}^{[\epsilon]}\right)^\chi$ (resp. $\left(\bar{U}_{p, \cO}^{[\epsilon]}\right)^\chi$) denote the $\chi$-isotypic component of $\left(\bar{U}_{p, \cO}^{[\epsilon]}\right)$ (resp. $\left(\bar{U}_{p, \cO}^{[\epsilon]}\right)$). Greenberg and Vatsal show that $\left(\bar{U}_{p, \cO}^{[\epsilon]}\right)^\chi$ has finite index in $\left(\cU_{p, \cO}^{[\epsilon]}\right)^\chi$, provided the assertions of Assumption \ref{main ass} are satisfied. 

\begin{proposition}\label{prop4.8} With respect to the notation above, suppose that the Assumption \ref{main ass} is satisfied. Then, the following assertions hold
\begin{enumerate}
    \item $\left(\bar{U}_{p, \cO}^{[\epsilon]}\right)^\chi$ has finite index in $\left(\cU_{p, \cO}^{[\epsilon]}\right)^\chi$,
    \item there is a short exact sequence of $\cO$-modules
    \[0\rightarrow H^1_{\op{nr}}(\Q, D)\rightarrow S_{\chi, \epsilon}(\Q)\rightarrow \Hom_{\cO[\Delta]}\left(\left(\cU_{p, \cO}^{[\epsilon]}\right)^\chi/\left(\bar{U}_{p, \cO}^{[\epsilon]}\right)^\chi, D\right)\rightarrow 0.\]
    Here, $H^1_{\op{nr}}(\Q, D)$ is the subgroup of $H^1(\Q,D)$ consisting of cohomology classes that are unramified at all primes.
\end{enumerate}
\end{proposition}
\begin{proof}
    The result follows from the proof of \cite[Proposition 3.1]{greenbergvatsalArtinrepns}.
\end{proof}

\begin{theorem}\label{conditions GV}
    Assume that  
     \begin{enumerate}[label=(\roman*)]
        \item the conditions of Assumption \ref{main ass} are satisfied.
        \item $H^0(\Delta_{\p}, D/D^{\epsilon_\p})=0$.
    \end{enumerate}
    Then, the following conditions are equivalent
    \begin{enumerate}
        \item $H^1_{\op{nr}}(\Q, D)=0$ and $\left(\cU_{p, \cO}^{[\epsilon]}\right)^\chi=\left(\bar{U}_{p, \cO}^{[\epsilon]}\right)^\chi$.
        \item $\mu_{\chi, \epsilon}=\lambda_{\chi, \epsilon}=0$.
         \item $S_{\chi, \epsilon}(\Q_\infty)=0$, 
    \end{enumerate}
\end{theorem}

\begin{proof}
    The result follows as a direct consequence of Corollary \ref{cor 4.7} and Proposition \ref{prop4.8}.
\end{proof}

\begin{lemma}\label{lemma4.10}
    With respect to the notation above, assume that $p\nmid \# \op{Cl}(K)$. Then, we find that $H^1_{\op{nr}}(\Q, D)=0$. 
\end{lemma}

\begin{proof}
    Since $p\nmid \# \Delta$, we find that $H^1(\Delta, D^{\op{G}_K})=0$, and thus from the inflation-restriction sequence, $H^1(\Q, D)$ injects into $\op{Hom}\left(\op{G}_K^{\op{ab}}, D\right)^\Delta$. Let $H_K$ be the Hilbert class field of $K$. We find that $H^1_{\op{nr}}(\Q, D)$ injects into $\op{Hom}(\op{Gal}(H_K/K), D)^\Delta$. Since $p\nmid \#\op{Cl}(K)$, we find that 
    \[\op{Hom}(\op{Gal}(H_K/K), D)=0,\] and hence, $H^1_{\op{nr}}(\Q, D)=0$.
\end{proof}

\begin{theorem}\label{cor 3.10}
 Assume that  
     \begin{enumerate}[label=(\roman*)]
        \item the conditions of Assumption \ref{main ass} are satisfied.
        \item $H^0(\Delta_{\p}, D/D^{\epsilon_\p})=0$,
        \item $K$ is $p$-rational,
        \item $p\nmid \#\op{Cl}(K)$.
    \end{enumerate}
 Then, $S_{\chi, \epsilon}(\Q_\infty)=0$.
\end{theorem}

\begin{proof}
   Since $p\nmid \# \op{Cl}(K)$, it follows from Lemma \ref{lemma4.10} that $H^1_{\op{nr}}(\Q, D)=0$. Let $G_p$ be the maximal pro-$p$ extension of $K$ which is unramified outside $p$. Since $p\nmid \# \op{Cl}(K)$, we find that $G_p^{\op{ab}}$ is isomorphic to $\cU_p/\bar{U}$. On the other hand, $K$ is $p$-rational if and only if the Leopoldt conjecture is true at $p$ and there is no $p$ torsion in $G_p^{\op{ab}}$ (cf. \cite{maire2020note} or \cite{movahhedi1990arithmetique}). This implies that 
   \[\left(\cU_{p, \cO}^{[\epsilon]}\right)^\chi/\left(\bar{U}_{p, \cO}^{[\epsilon]}\right)^\chi=0.\] The results then follow from Theorem \ref{conditions GV}.
\end{proof}
\subsection{Special cases}
\par In this subsection, we consider certain special cases. First, we consider some $2$-dimensional Artin representations of dihedral type. Let $(L, \zeta)$ be a pair, where $L$ is an imaginary quadratic extension of $\Q$ and $\zeta:\op{G}_L\rightarrow \bar{\Q}^\times$ is a character. Consider the $2$-dimensional Artin representation 
\[\rho=\rho_{(L, \zeta)}:=\op{Ind}_{\op{G}_L}^{\op{G}_\Q}\zeta.\] When restricted to $L$, the representation $\rho$ decomposes into a direct sum of characters 
\[\rho_{|\op{G}_L}=\mtx{\zeta}{}{}{\zeta'}.\] Here, $\zeta'$ is given as follows
\[\zeta'(x)=\zeta(c x c^{-1}), \] where $c$ denotes the complex conjugation. We remark that the representation $\rho$ is a self-dual representation of dihedral type if and only if $\zeta'=\zeta^{-1}$. Let $p$ be an odd prime number which splits in $L$ as a product $p\cO_L=\p \p^*$. We set $\epsilon_\p:=\zeta_{|\op{G}_\p}$ and note that $\epsilon_{\p^*}=\zeta'_{|\op{G}_\p}$. The field $K=L(\zeta, \zeta')$ is the extension of $L$ generated by $\zeta$ and $\zeta'$. Choose an extension $\mathcal{F}/\Q_p$ containing $\Q_p(\chi, \epsilon_\p)$ and set $\cO$ to denote its valuation ring, and $D$ the associated $\cO$-divisible module. With respect to such a choice, we let $S_{\chi, \epsilon}(\Q_\infty)$ be the associated the Selmer group over the cyclotomic $\Z_p$-extension.
\begin{theorem}\label{thm 4.12}
    With respect to the notation above, assume that the following conditions hold
    \begin{enumerate}[label=(\roman*)]
        \item $\rho$ is irreducible,
        \item the order of $\zeta$ is coprime to $p$,
        \item $\epsilon_\p$ is nontrivial and $\epsilon_\p\neq \epsilon_{\p^*}$. 
    \end{enumerate}
    Then, the following conditions are equivalent
    \begin{enumerate}
        \item $H^1_{\op{nr}}(\Q, D)=0$ and $\left(\cU_{p, \cO}^{[\epsilon]}\right)^\chi=\left(\bar{U}_{p, \cO}^{[\epsilon]}\right)^\chi$.
        \item $\mu_{\chi, \epsilon}=\lambda_{\chi, \epsilon}=0$.
         \item $S_{\chi, \epsilon}(\Q_\infty)=0$.
    \end{enumerate}
    If $p\nmid h_K$ and $K$ is $p$-rational, then, the above conditions are satisfied. 
\end{theorem}
\begin{proof}
    The result follows from Theorem \ref{conditions GV} once we show that
    \begin{enumerate}[label=(\roman*)]
        \item the conditions of Assumption \ref{main ass} are satisfied.
        \item $H^0(\Delta_{\p}, D/D^{\epsilon_\p})=0$.
    \end{enumerate}
    Consider the conditions for Assumption \ref{main ass} to hold.
    \begin{enumerate}
        \item The field $K$ is the extension of $L$ cut out by $\zeta$ and $\zeta'$. Since the order of $\zeta$ is assumed to be coprime to $p$, the same holds for $\zeta'$. Since $p$ is odd, it follows that $p\nmid [K:\Q]$. 
        \item Since $K$ is an imaginary quadratic field complex conjugation acts via the matrix $\mtx{}{1}{1}{}$, whose eigenvalues are $1$ and $-1$ respectively. Therefore, $d^+=d^-=1$. 
        \item The third condition follows from the assumption that $\epsilon_\p$ is nontrivial and $\epsilon_\p\neq \epsilon_{\bar{\p}}$. 
    \end{enumerate}
    Finally, we note that since $\epsilon_{\bar{\p}}$ is nontrivial, $H^0(\Delta_\p, D/D^{\epsilon_\p})=0$. 

    \par By the proof of Theorem \ref{cor 3.10}, we note that $H^1_{\op{nr}}(\Q, D)=0$ and $\left(\cU_{p, \cO}^{[\epsilon]}\right)^\chi=\left(\bar{U}_{p, \cO}^{[\epsilon]}\right)^\chi$ if $p$ is such that $p\nmid h_K$ and $K$ is $p$-rational, then, the above conditions are satisfied. 
\end{proof}

\par Next, we consider the group of rotational symmetries of the icosahedron. This group is simply $A_5$. An orthogonal set is a set of $6$ points, so that three pairwise orthogonal lines can be drawn between pairs of them. The midpoints of the $30$ edges of an icosahedron can be partitioned into $5$ orthogonal sets such that these sets are permuted by $A_5$. Let $g\in A_5$ be an element that corresponds to rotation of the icosahedron with axis $(x,y,z)$ and angle $\theta$. Then the corresponding matrix is 
\[ \left( {\begin{array}{ccc}
 \cos \theta+(1-\cos \theta)x^2 & (1-\cos \theta)xy-z\sin \theta & (1-\cos \theta)xz+y\sin \theta\\
 (1-\cos \theta)xy+z\sin \theta & \cos \theta+(1-\cos \theta)x^2 &  (1-\cos \theta)yz-x\sin \theta \\
 (1-\cos \theta)xz-y\sin \theta & (1-\cos \theta)yz+x\sin \theta &  \cos \theta+(1-\cos \theta)z^2. 
 \end{array} } \right)\]
This gives a representation 
\[r:A_5\rightarrow \op{GL}_3(\bar{\Q}).\] Let $K/\Q$ be a Galois extension with Galois group $\op{Gal}(K/\Q)$ isomorphic to $A_5$, and $\varrho$ be the composite 
\[\varrho:\op{G}_\Q\rightarrow \op{Gal}(K/\Q)\xrightarrow{\sim} A_5\xrightarrow{r}\op{GL}_3(\bar{\Q}),\] where the initial map is the quotient map.
 Any $5$-cycle corresponds to a rotation by an angle of $\frac{2\pi}{5}$, and therefore has eigenvalues $e^{\frac{2\pi i}{5}}, e^{-\frac{2\pi i}{5}}, 1$. Let $g\in A_5$ be a $5$-cycle and $D$ be the group generated by $g$, and $L:=K^D$ be the field fixed by $D$. Let $p\geq 7$ be a prime which is split in $L$ and is inert in $K/L$. Then, the restriction to the decomposition group at $p$ is of the form $\op{diag}\left(\alpha_p, \alpha_p^{-1}, 1\right)$, where $\alpha_p$ is an unramified character of order $5$. Let $\psi:\op{G}_\Q\rightarrow \cO^\times$ be an even character which is ramified at $p$, and set $\rho:=\varrho\otimes \psi$.

 \begin{theorem}\label{example thm 2}
     With respect to the above notation, the following conditions are equivalent
\begin{enumerate}
        \item $H^1_{\op{nr}}(\Q, D)=0$ and $\left(\cU_{p, \cO}^{[\epsilon]}\right)^\chi=\left(\bar{U}_{p, \cO}^{[\epsilon]}\right)^\chi$.
        \item $\mu_{\chi, \epsilon}=\lambda_{\chi, \epsilon}=0$.
         \item $S_{\chi, \epsilon}(\Q_\infty)=0$.
    \end{enumerate}
    If $p\nmid h_K$ and $K$ is $p$-rational, then, the above conditions are satisfied. 
 \end{theorem}

\begin{proof}
As in the proof of Theorem \ref{thm 4.12}, the result follows from Theorem \ref{conditions GV} once we show that
    \begin{enumerate}[label=(\roman*)]
        \item the conditions of Assumption \ref{main ass} are satisfied.
        \item $H^0(\Delta_{\p}, D/D^{\epsilon_\p})=0$.
    \end{enumerate}
    Let us verify the conditions of the Assumption \ref{main ass}.
    \begin{enumerate}
        \item Since it is assumed that $p\geq 7$, it follows that $p\nmid \# A_5$, and hence, $p\nmid [K:\Q]$.
        \item Let $v$ be an archidean prime, and $\sigma_v$ be a generator of the decomposition group at $v$. Then, $\rho(\sigma_v)=\varrho(\sigma_v) \psi(\sigma_v)=\varrho(\sigma_v)$. Since $\varrho$ only gives rise to rotational symmetries and $\varrho(\sigma_v)$ has order $2$, it must correspond to the rotation with $\theta=\pi$. This means that it is conjugate to $\op{diag}(1,-1,-1)$, and hence, $d^+=1$.
        \item For the third condition, it has been arranged that $\varrho_{|\Delta_\p}=\op{diag}(\alpha_p, \alpha_p^{-1}, 1)$, where $\alpha_p$ is an unramified character of order $5$, and hence, $\epsilon_p:=\alpha_p \psi$ is a ramified character which occurs with multiplicity $1$.
    \end{enumerate}
    Finally, we note that $H^0(\Delta_\p, D/D^{\epsilon_p})=0$ since the characters $\alpha_p^{-1}\psi$ and $\psi$ are nontrivial. This is because $\psi$ is ramified at $p$, while $\alpha_p$ is not. 
\end{proof}

\section{Distribution questions for Artin representations}\label{s 4}

\subsection{Iwasawa theory of class groups}
\par In this section, we introduce and study some natural distribution questions for Iwasawa invariants of class groups. These discussions serve to motivate similar questions for Artin representations, which we discuss in the next subsection. Given a number field $K$ and a prime number $p$, the Iwasawa invariants $\mu_p(K)$ and $\lambda_p(K)$ are natural invariants to consider associated with the growth of $p$-primary parts of class groups in the cyclotomic extension of $K$. One considers the following question.
\begin{question}
    Given an imaginary quadratic $K/\Q$, how does $\lambda_p(K)$ vary as $p\rightarrow \infty$? In other words, for a given $\lambda\in \Z_{\geq 0}$, what can be said about the upper and lower densities of the set of primes $p$ for which $\lambda_p(K)=\lambda$.
\end{question}
We call the above set $\cF_{ \lambda}$, and its upper (resp. lower) density $\bar{\mathfrak{d}}_{\lambda}$ (resp. $\underline{\mathfrak{d}}_{\lambda}$).
We note that if $p$ is inert or ramified in $K$ and $p\nmid h_K$, then, $h_p(K_n)=0$ for all $n$. In particular, $\lambda_p(K)=0$. This implies in particular that $\underline{\mathfrak{d}}_{0}\geq  \frac{1}{2}$. On the other hand, if $p$ splits in $K$, then, $r_{p,K}\geq 1$ and thus, $\lambda_p(K)\geq 1$. This implies that $\mathfrak{d}_0=\frac{1}{2}$. Let $p$ be a prime which splits in $K$ for which $p\nmid h_K$. The former condition is satisfied by $\frac{1}{2}$ of the primes and the latter condition is satisfied for all but finitely many primes. The following result is a corollary to Gold's criterion. 

\begin{corollary}
    Let $p$ be an odd prime number which splits into $\p \p^*$ in $\cO_K$. Suppose that $r>1$ is an integer not divisible by $p$ and such that $\p^r=(\alpha)$. Then, the following conditions are equivalent
    \begin{enumerate}
        \item $\lambda_p(K)>1$,
        \item $\op{Tr}(\alpha)^{p-1}\equiv 1\mod{p^2}$. 
    \end{enumerate}
\end{corollary}
 Indeed we find that for each number $a_0\in [1, p-1]$, there is precisely one congruence class $a$ modulo-$p^2$ such that:
 \begin{itemize}
     \item $a\equiv a_0\pmod{p}$, and
     \item $a^{p-1}\equiv 1\mod{p^2}$.
 \end{itemize} Therefore, the probability that an integer $a\in \Z/p^2\Z$ satisfies the congruence $a^{p-1}\equiv 1\mod{p^2}$ is $\frac{p-1}{p^2}=\frac{1}{p}-\frac{1}{p^2}$. Let $N_K(X)$ be the number of primes $p\leq X$ such that $p$ is split in $K$ and $\lambda_p(K)>1$. Thus, the heuristic suggests that 
 \[N_K(X)\sim \sum_{p\in \Omega'(X)} \left(\frac{1}{p}-\frac{1}{p^2}\right)\sim \log \log X,\] where $\Omega'(X)$ is the set of primes $p\leq X$ that split in $K$. This leads us to the following expectation.
 \begin{conjecture}
     Let $M_p(X)$ be the number of primes $\leq X$ that split in $K$ for which $\lambda_p(K)>1$, then, $M_p(X)= \frac{1}{2}\log \log X+O(1)$.
 \end{conjecture}
 In particular the set of primes $p$ for which $\lambda_p(K)>1$ is infinite of density $0$, and thus, the conjecture in particular predicts that $\mathfrak{d}_{1}=\frac{1}{2}$. Horie \cite{horie1987note} has proven the infinitude of primes $p$ such that $\lambda_p(K)>1$ and Jochnowitz \cite{jochnowitz1994p} on the other hand the infinitude of primes $p$ for which $\lambda_p(K)=1$.

 \par On the other hand, one can fix a prime and study the variation of $\lambda_p(K)$ as $K$ ranges over all imaginary quadratic fields. We separate this problem into two cases, namely that of imaginary quadratic fields in which $p$ is inert, and those in which $p$ splits. For the primes $p$ that are inert in $K$ Ellenberg, Jain and Venkatesh \cite{ellenberg2011modeling} make the following prediction based on random matrix heuristics. 
 \begin{conjecture}[Ellenberg, Jain, Venkatesh]
     Amongst all imaginary quadratic fields $K$ in which $p$ is inert, the proportion for which $\lambda_p(K)=r$ is equal to 
     \[p^{-r}\prod_{t>r}\left(1-p^{-t}\right).\]
 \end{conjecture}
 We note that $\lambda_p(K)=0$ if and only if $p\nmid h_K$. The probability that $p\nmid h_K$ is, according to the Cohen-Lenstra heuristic, predicted to be equal to 
 \[\prod_{t>0}\left(1-p^{-t}\right).\]

\subsection{Artin representations}
\par Given an Artin representation $\rho$, we are interested in understanding the variation of $\mu$ and $\lambda$-invariants as $p$ ranges over all prime numbers. We specialize our discussion to odd $2$-dimensional Artin representations of dihedral type
\[\rho=\op{Ind}_{\op{G}_L}^{\op{G}_\Q} \zeta,\] where $L$ is an imaginary quadratic field and $\zeta:\op{G}_L\rightarrow \bar{\Q}^\times$ is a character. Let $\zeta'$ be the character defined by $\zeta'(\sigma)=\zeta(c \sigma c^{-1})$, where $c\in \op{G}_\Q$ denotes the complex conjugation. Here, $L/\Q$ is an imaginary quadratic field. Since it is assumed that $\rho$ is of dihedral type, it follows that $\zeta'=\zeta^{-1}$ and the extension $K$ is simply the extension of $L$ that is fixed by the kernel of $\zeta$. Furthermore, assume that $2\nmid [K:L]$. Let $S(\rho)$ be the set of primes $p$ such that 
\begin{enumerate}
    \item $p$ is odd and $p\nmid [K:\Q]$, 
    \item $p$ splits in $L$, $p\cO_L=\pi \pi^*$, and $\pi$ is inert in $K/L$.
\end{enumerate}
For $p\in S(\rho)$, it is then clear from the assumptions that the Assumption \ref{main ass} holds, and that $H^0(\Delta_\p, D/D^{\epsilon_\p})=0$ for any of the primes $\p|p$ of $K$ and the character $\epsilon_\p=\zeta_{|\Delta_\p}$. Let $\epsilon$ be the character associated to this choice of $\p$ and $\epsilon_\p$. 

\begin{definition}
    Let $T(\rho)$ be the set of primes $p\in S(\rho)$ such that the Selmer group $S_{\chi, \epsilon}(\Q_\infty)$ vanishes for all of the primes $\p$ that lie above $p$. Set $T'(\rho):=S(\rho)\backslash T(\rho)$. 
\end{definition}

We note that the Theorem \ref{cor 3.10} implies that for $p\in S(\rho)$ such that $K$ is $p$-rational, then, $p\in T(\rho)$. 

\par Let us recall some heuristics on the $p$-rationality condition. The $\Z_p$-rank of $\cU_p$ is $n=[K:\Q]=r_1+2r_2$. On the other hand, the rank of $U_K$ is $k:=r_1+r_2-1$. The probabilty that a random matrix with $n$ columns and $k$ rows does not have full $\Z_p$-rank is 
\[\op{Pr}_{k,n}:=1-\frac{\prod_{i=0}^{k-1}\left(p^n-p^i\right)}{p^n}\leq \frac{1}{p^{n-k+1}}+\frac{1}{p^{n-k+2}}+\dots+\frac{1}{p^n}.\]
Therefore, the expected number of primes $p$ at which $K$ is not $p$-rational is finite, since according to this heuristic,
\[\sum_{p} \op{Pr}_{k,n}\leq  \sum_{p} \left(\frac{1}{p^{r_2+2}}+\frac{1}{p^{r_2+3}}+\dots+\frac{1}{p^n}\right)=\zeta(r_2+2)+\zeta(r_2+3)+\dots+\zeta(n)<\infty.\]
This is indeed a conjecture due to Gras.

 \begin{conjecture}[Gras]
     Let $K$ be a number field. Then for all large values of $p$, $K$ is $p$-rational.
 \end{conjecture}

 This leads us to make the following conjecture.

  \begin{conjecture}
     With respect to above notation, the set of primes $T'(\rho)$ is finite. Thus, there are only finitely many pairs $(\p, \epsilon)$ where $\p$ is a prime above $p\in S(\rho)$ and $\epsilon:\op{G}_{L_\p}\rightarrow \bar{\Q}_p^\times$ is a character, such that the associated Selmer group $S_{\chi, \epsilon}(\Q_\infty)\neq 0$. 
 \end{conjecture}

 There is a relationship between Gras' conjecture and the generalized abc-conjecture, stated below, cf. \cite{vojtaabc}.
 \begin{conjecture}[Generalized abc-conecture]
     Let $K$ be a number field and $I$ be an ideal in $\cO_K$, the radical of $I$ is defined as follows
     \[\op{Rad}(I):=\prod_{\p|p} N(\p),\]where the product is over all prime ideals $\p$ dividing $I$, and $N(\p):=\# \left(\cO_K/\p\cO_K\right)$ is the norm of $\p$. The generalized abc conjecture predicts that for any $\epsilon>0$, there exists a constant $C_{K, \epsilon}>0$ such that 
     \[\prod_v \op{max}\{|a|_v, |b|_v, |c|_v\}\leq C_{K, \epsilon} \left(\op{Rad}(abc)\right)^{1+\epsilon},\]
     holds for all non-zero $a,b,c\in \cO_K$ such that $a+b=c$.
 \end{conjecture}Let us recall a recent result of Maire and Rougnant, cf. \cite[Theorem A]{maire2020note}.

\begin{theorem}[Maire-Rougnant]\label{maire rougnant input}
    Let $K/\Q$ be an imaginary $S_3$ extension. Then, the generalized abc-conjecture for $K$ implies that there is a constant $c>0$ such that
    \[\# \{p\leq x\mid K \text{ is }p\text{-rational}\}\geq c\log x.\]
\end{theorem}
 The above result has implication to the vanishing of Iwasawa modules (and invariants). 

\begin{theorem}\label{last theorem}
    Let $K/\Q$ be an imaginary $S_3$ extension and $\rho$ be a $2$-dimensional Artin representation that factors through $\op{Gal}(K/\Q)$. Then,
    \[\# \{p\leq x\mid p\in T(\rho)\}\geq c\log x.\]
\end{theorem}
\begin{proof}
    Theorem \ref{cor 3.10} implies that for $p\in S(\rho)$ such that $K$ is $p$-rational, then, $p\in T(\rho)$. The result thus follows.
\end{proof}

\bibliographystyle{alpha}
\bibliography{references}
\end{document}